\documentclass[12pt]{article}
\usepackage{enumitem}
\usepackage{amsmath}
\usepackage{amssymb}
\usepackage{booktabs}
\usepackage{setspace}
\usepackage{graphicx}
\usepackage{amsthm}
\usepackage{color}
\usepackage{float}
\usepackage{subcaption}
\usepackage[authoryear,round]{natbib}
\makeatletter 
\renewcommand\NAT@biblabel[1]{[#1]} 
\usepackage[affil-it]{authblk}

\usepackage{geometry}
\usepackage{authblk}

\newtheorem{thm}{Theorem}
\newtheorem{lem}{Lemma}

\newtheorem{as}{Assumption}

\newtheorem{rem}{Remark}

\newcommand{\Sigmahat}{\hat{\Sigma}}
\newcommand{\fn}{\frac{1}{n}}
\newcommand{\tr}{\textup{tr}}
\newcommand{\pr}{\textup{P}}

\newcommand{\gcv}{\textup{GCV}}
\newcommand{\rank}{\textup{rank}}
\newcommand{\expect}{\textup{E}}
\newcommand{\diag}{\textup{diag}}
\newcommand{\real}{\mathbb{R}}
\newcommand{\betahat}{\hat{\beta}}
\newcommand{\lambdahat}{\hat{\lambda}}

\newcommand*{\transpose}{%
  {\mathpalette\@transpose{}}%
  }
  
  \expandafter\let\expandafter\oldproof\csname\string\proof\endcsname
\let\oldendproof\endproof
\renewenvironment{proof}[1][\proofname]{%
  \oldproof[\bfseries \scshape #1]%
}{\oldendproof}

\title{Uniform Consistency of Generalized Cross-Validation for Ridge Regression in High-Dimensional Misspecified Linear Models}

\author{Akira Shinkyu\thanks{Email: akira-shinkyu@biwako.shiga-u.ac.jp.}}

\providecommand{\keywords}[1]{\textbf{Keywords}: #1}

\affil{Shiga University\thanks{Faculty of Economics, Shiga University, 1-1-1 Banba, Hikone, Shiga, 522-8522, Japan.}}

\date{\today}

\allowdisplaybreaks

\begin{document}
\maketitle
\begin{abstract}
This study examines generalized cross-validation for the tuning parameter selection for ridge regression in high-dimensional misspecified linear models. The set of candidates for the tuning parameter includes not only positive values but also zero and negative values. We demonstrate that if the second moment of the specification error converges to zero, generalized cross-validation is still a uniformly consistent estimator of the out-of-sample prediction risk. This implies that generalized cross-validation selects the tuning parameter for which ridge regression asymptotically achieves the smallest prediction risk among the candidates if the degree of misspecification for the regression function is small. Our simulation studies show that ridge regression tuned by generalized cross-validation exhibits a prediction performance similar to that of optimally tuned ridge regression and outperforms the Lasso under correct and incorrect model specifications.
\end{abstract}
\keywords{Ridge regression, high-dimensional data, prediction, cross-validation, model misspecification.}

\section{Introduction}
Ridge regression is one of the most well-known estimation methods in statistics. It was proposed by \cite{hk1970} to address multicollinearity between covariates and is also known as $\ell_2$ regularization. Ridge regression estimators are biased toward zero but have smaller variances than ordinary least squares estimators, and they admit closed-form expressions even in high-dimensional settings. Extensive reviews of ridge regression are presented by \cite{h2020}. \par

In this study, we consider the selection of the tuning parameter in ridge regression for high-dimensional prediction. The prediction performance of ridge regression depends on its tuning parameter. When $n>p$ for the sample size $n$ and the number of covariates $p$, the optimal tuning parameter for ridge regression is always positive. However, recent studies have shown that zero or even negative tuning parameters can be optimal when $p>n$ (\citealt{kls2020}, \cite{tb2023}). It has also been revealed that signs of the optimal tuning parameter relate to the alignment of regression coefficients and the eigenvectors of covariances (\cite{hmrt2022}, \cite{kls2020}, \cite{wu2020}). However, as we do not know the optimal tuning parameter in practice, we need to rely on data-driven procedures to tune ridge regression. \par

 Cross-validation procedures are the most popular methods for selecting the tuning parameter for ridge regression. Among these procedures, we mainly consider generalized cross-validation (GCV) in this study, originating from \cite{cw1978} and \cite{ghw1979}. GCV for ridge regression is computationally efficient because its objective function is relatively simple. Moreover, it serves as an approximation of leave-one-out cross-validation (LOOCV), which is more computationally intensive. The optimality of GCV in the sense of in-sample prediction for ridge regression when $n>p$ was investigated in \cite{l1985}, \cite{l1986}, and \cite{l1987}, for example. \par 

Recently, \cite{hmrt2022} showed that GCV is a uniformly consistent estimator of the out-of-sample prediction risk even when $p>n$ holds true under identity covariances and zero-mean spherical random regression coefficients. \cite{pwrt2021} obtained the uniform consistency of GCV under more general settings than  \cite{hmrt2022}. \cite{pwrt2021} allow that covariances can have general structures and regression coefficients can be nonrandom. Furthermore, \cite{pwrt2021} allow that the candidate set of the tuning parameter can include not only positive values but also zero and negative values. \par 

However, these previous studies assume that regression functions are correctly specified as linear functions when GCV is analyzed for high-dimensional ridge regression. In general, the forms of true regression functions are unknown. Therefore, it may be highly possible that regression functions are incorrectly specified as linear functions in practice.\par 

To address the misspecification problem, we investigate GCV for the tuning parameter selection for high-dimensional ridge regression when regression functions are incorrectly specified as linear functions. We show that if the second moment of the specification error converges to zero, GCV is still a uniformly consistent estimator of the out-of-sample prediction risk. This implies that GCV selects, from the candidates that include not only positive values but also zero and negative values, the tuning parameter for which ridge regression asymptotically achieves the smallest prediction risk under the misspecification. By combining our theoretical results with Theorem 4.2 from \cite{pwrt2021}, the uniform consistency also holds for LOOCV. \par 
Moreover, we conduct simulation studies to assess the prediction performance of ridge regression tuned by GCV under correct and incorrect model misspecifications. Simulation results show that ridge regression tuned by GCV exhibits a prediction performance similar to that of optimally tuned ridge regression and considerably outperforms the Lasso (\citealt{t1996}) regardless of the model specifications. Similar results are obtained for ridge regression tuned by LOOCV. \par 

We note that \cite{prt2022} also showed the uniform consistency of GCV for high-dimensional ridge regression where regression functions can be nonlinear. They do not even assume specific models; however, they require the bounded moment of order larger than four for the response variable. On the other hand, although we assume that the second moment of the specification error converges to zero, we require only the bounded second moment for the response variable. Thus, our uniform consistency result covers some regression functions which do not satisfy the assumption of \cite{prt2022}.\par 
   
The remainder of this manuscript is organized as follows. In Section 2, we formally introduce a misspecified linear model, ridge regression, GCV, and the out-of-sample prediction risk. In Section 3, we present the assumptions and the uniform consistency result for GCV under the misspecification. In Section 4, we present the results of simulation studies for ridge regression tuned by GCV and other prediction methods. Section 5 concludes this article. The appendix provides proofs of the theoretical results.

\section{Model, GCV, and Prediction Risk}
We first introduce some notations. Let $a^T$ and $\|a\|$ be the transpose and $\ell_2$ norm of a vector $a$, respectively. Let $\lambda_{\min}(A)$ and $\lambda_{\max}(A)$ be the minimum and maximum eigenvalues of a matrix $A$, respectively. Let $\tr(A)$ be the trace of a matrix $A$. Let $I_m \in \real^{m \times m}$ be the identity matrix for a positive integer $m$. \par 
The true regression model is given by
\begin{equation}
    y_i=f(x_i)+\epsilon_i, \label{101}
\end{equation}
where $y_i\in \mathbb{R}$ is a response variable, $x_i=(x_{i,1},\ldots,x_{i,p})^T\in \mathbb{R}^{p}$ is a covariate vector, $f: \mathbb{R}^p \rightarrow \mathbb{R}$ is an unknown regression function, and $\epsilon_i \in \real$ is an unobservable error term. Throughout this study, we assume that $(x_1,y_1),\ldots, (x_n,y_n)$ are independently and identically distributed and consider a proportional asymptotics: $p/n \rightarrow \gamma \in (0, \infty)$ satisfying $\gamma\neq 1$. \par 

We assume that the regression function $f(x_i)$ is incorrectly specified as a linear function. That is, we incorrectly use the following linear model:
\begin{equation}
    y_i=x_i^T\beta_0+u_i, \label{84}
\end{equation}
where $\beta_0\in \mathbb{R}^p$ is an unknown nonrandom coefficient vector, $u_i \in \real$ is a new unobservable error term that is given by $u_i=\delta_i+\epsilon_i$, and $\delta_i=f(x_i)-x_i^T\beta_0$ is a specification error. \par

We introduce a matrix representation of the misspecified linear model \eqref{84}. Let $Y=(y_1,\ldots, y_n)^T \in \mathbb{R}^n$, $X=(x_1, \ldots, x_n)^T \in \mathbb{R}^{n \times p}$, and $u=(u_1,\ldots,u_n)^T\in \mathbb{R}^n$. Then,
a matrix form of the misspecified linear model \eqref{84} is given by
\begin{equation*}
    Y=X\beta_0 +u.
\end{equation*}
We can write $u=\delta+\epsilon$, where $\delta=(\delta_1,\ldots, \delta_n)^T\in \mathbb{R}^n$ and $\epsilon=(\epsilon_1,\ldots,\epsilon_n)^T\in \mathbb{R}^n$. We also assume that $X$ is full row rank: $\rank(X)=n$.\par

We estimate the coefficient of the misspecified linear model $\beta_0$ using ridge regression. Let $\lambda$ be a tuning parameter for ridge regression. The ridge estimator $\betahat_\lambda$ is given by
\begin{equation*}
    \betahat_\lambda=\arg\min_{\beta \in \mathbb{R}^p}\frac{1}{n}\|Y-X\beta\|^2+\lambda\|\beta\|^2.
\end{equation*}
We can write $\betahat_\lambda$ explicitly as
\begin{equation}
    \betahat_\lambda=(\Sigmahat +\lambda I_p)^{-1}\frac{1}{n}X^TY, \label{125}
\end{equation}
where $\Sigmahat=X^TX/n$. However, the equation in \eqref{125} does not allow us to choose $\lambda=0$ because $\Sigmahat$ is not invertible when $p>n$. Therefore, to extend the range of $\lambda$ to include zero as in \cite{pwrt2021}, we redefine the ridge estimator as 
\begin{equation}
    \betahat_\lambda=(\Sigmahat+\lambda I_p)^{+}\fn X^TY, \label{118}
\end{equation}
where $A^+$ is the Moore--Penrose pseudoinverse of a matrix $A$. As long as there are no zero eigenvalues for $\Sigmahat+\lambda I_p$, no difference exists between \eqref{125} and \eqref{118}. When $\lambda=0$, $\betahat_\lambda$ in \eqref{118} is equal to the minimum $\ell_2$ norm (min-norm) least squares estimator, which yields zero residuals when $X$ is full row rank.  \par

We use GCV to choose the tuning parameter for ridge regression. To introduce the objective function for GCV, let $S_\lambda$ be the ridge smoother matrix, which is defined as 
\begin{equation*}
    S_\lambda=X(\Sigmahat+\lambda I_p)^{+}\fn X^T.
\end{equation*}
Note that $S_\lambda Y=X\betahat_\lambda$. Then, the objective function of GCV can be written as 
\begin{equation}
\gcv(\lambda)=\frac{1}{n}\sum_{i=1}^n\left(\frac{y_i-x_i^T\betahat_\lambda}{1-\fn \tr(S_\lambda)}\right)^2. \label{127}
\end{equation}
However, if we plug $\lambda=0$ into $\gcv(\lambda)$ in \eqref{127}, we obtain $0/0$ because residuals of $\betahat_0$ are zero and $S_0=I_n$ when $X$ is full row rank. Therefore, as in \cite{hmrt2022} and \cite{pwrt2021}, we define $\gcv(0)$ as the limit of $\gcv(\lambda)$ as $\lambda \rightarrow 0$. It is given by
\begin{equation*}
\gcv(0):=\lim_{\lambda\rightarrow 0}\gcv(\lambda)=\frac{1}{n}\frac{Y^T(\fn XX^T)^{+2}Y}{(\fn \tr(\Sigmahat^+))^2}.
\end{equation*}

Let $\Lambda$ be a closed interval that is the set of candidates for $\lambda$. The endpoints of $\Lambda$ are defined in Section 3. The tuning parameter selected through GCV from $\Lambda$ is given by
\begin{equation*}
    \hat{\lambda}=\arg\min_{\lambda \in \Lambda}\gcv(\lambda).
\end{equation*}
\par 
We next introduce the optimal tuning parameter for ridge regression. As in \cite{pwrt2021}, we use the conditional out-of-sample prediction risk as a measure of the prediction accuracy of ridge regression. The prediction risk is given by 
\begin{equation*}
    R(\lambda)=\expect[(y_0-x_0^T\betahat_{\lambda})^2|X,Y].
\end{equation*}
The test data $(x_0,y_0)$ are independent of the training data $(X,Y)$ and follow the same distribution as the sample $(x_i,y_i),\ i=1,\ldots,n$. The optimal tuning parameter for ridge regression is given by the minimizer of the prediction risk from the candidate set $\Lambda$. The optimal tuning parameter $\lambda^*$ is given by  
\begin{equation*}
    \lambda^{*}=\arg\min_{\lambda \in \Lambda} R(\lambda).
\end{equation*}
\par Note that $\lambdahat$ and $\lambda^*$ are random variables because they depend on the training data $(X,Y)$. If GCV and the prediction risk have multiple minimizers, we denote a minimizer of GCV by $\lambdahat$ and a minimizer of the prediction risk by $\lambda^*$.\par

Our purpose is to show that GCV is a uniformly consistent estimator of the prediction risk under the misspecified linear model. The uniform consistency of GCV implies the asymptotic equivalence between $R(\lambdahat)$ and $R(\lambda^*)$, which means GCV tuning is asymptotically optimal tuning for ridge regression under the misspecification.

\section{Theoretical Results}
To analyze the relationship between GCV and the prediction risk under the misspecification, we present some assumptions.
\begin{as}
\textup{$\epsilon_i$ is independent of $x_i$ and has a mean of zero, a variance of $\sigma^2$, and a finite $4+\eta$ moment for some $\eta>0$.}
\end{as}
\begin{as}
\textup{$x_i$ is expressed as $x_i=\Sigma^{1/2}z_i$, where $\Sigma \in \mathbb{R}^{p \times p}$ is a positive definite matrix, and $z_i$ has a mean of zero, a variance of one, and a finite $4+\eta$ moment for some $\eta>0$.}
\end{as}
\begin{as}
\textup{There exist positive constants $C_{\min}$ and $C_{\max}<\infty$ such that $C_{\min}\leq \lambda_{\min}(\Sigma)\leq \lambda_{\max}(\Sigma)\leq C_{\max}$. }
\end{as}
These assumptions are the same as Assumptions 1--3 in \cite{pwrt2021}, except for the specification of the regression function. Under these assumptions and the correct specification, \cite{pwrt2021} showed the uniform consistency of GCV. We rely on their results to evaluate the linear components of GCV and the prediction risk. Note that these assumptions are also standard in the literature for high-dimensional ridge regression.  \par 
Next, we introduce an assumption about the coefficient vector of the misspecified linear model and the regression function.
\begin{as}
\textup{$\beta_0$ is defined as 
\begin{equation*}
    \beta_0=\arg\min_{\beta \in \real^p}\expect[\{f(x_i)-x_i^T\beta\}^2],
\end{equation*}
and there exists a finite positive constant $C_f$ such that $\expect[f(x_i)^2]\leq C_f$. Furthermore, $\delta_i$ satisfies
\begin{equation*}
    \expect(\delta_i^2)=o(1).
\end{equation*}
}
\end{as}
Note that $\|\beta_0\|^2\leq C_f C_{\max}/C_{\min}^2$ holds under Assumptions 2--4. Because \cite{pwrt2021} assume that the regression function is correctly specified as a linear function, no specification error exists and $\expect(\delta_i^2)=0$ holds in their study. On the other hand, Assumption 4 is more general than their assumption because the specification error does not need to be zero for fixed $n$. For example, if $x_i \sim N(0, I_p)$, the regression function
\begin{equation}
    f(x_i)=x_i^T\beta+p 1(x_{i,1}\leq q(p)) \label{197}
\end{equation}
satisfies Assumption 4, where $1(\cdot)$ is the indicator function, $\beta$ is some nonrandom vector satisfying $\|\beta\|\leq C_\beta$ for some finite positive constant $C_\beta$, and $q(p)$ is the quantile such that $\pr(x_{i,1}\leq q(p))=1/p^3$.  $\expect[f(x_i)^4]$ cannot be bounded by a finite positive constant because it includes $p$. Therefore, the response variable with this regression function does not satisfy Assumption 2 of \cite{prt2022}.\par

We define $\Lambda=[\lambda_{\min},\lambda_{\max}]$, where $\lambda_{\min}=-\kappa^2 C_{\min}(1-\sqrt{\gamma})^2$ for some constant $0<\kappa<1$ and $\lambda_{\max}$ is a finite positive constant. The lower endpoint $\lambda_{\min}$ is required to show that the minimum positive eigenvalue of $\Sigmahat+\lambda I_p$ is bounded away from zero. The upper endpoint $\lambda_{\max}$ is also required to show that the maximum eigenvalue of $\Sigmahat+\lambda I_p$ has an upper bound. The boundedness of these eigenvalues is important for the asymptotic analysis of GCV and the prediction risk. \par

To show the uniform consistency of GCV under the misspecification, we need to obtain the pointwise convergence result for GCV. We also consider the same approaches as did \cite{pwrt2021} to obtain it.  That is, we decompose the objective function of GCV and the prediction risk into bias components, variance components, and cross components:
\begin{align*}
    R(\lambda)&=R_b(\lambda)+R_v(\lambda)+R_c(\lambda),\\
    \gcv(\lambda)&=\frac{\gcv_b(\lambda)}{\gcv_d(\lambda)}+\frac{\gcv_v(\lambda)}{\gcv_d(\lambda)}+\frac{\gcv_c(\lambda)}{\gcv_d(\lambda)}.
\end{align*}
These terms are specifically defined in the appendix. First, we show that these cross components are asymptotically negligible. Next, we show that the bias components and variance components of the prediction risk and GCV are asymptotically equivalent, respectively. These results imply the pointwise convergence of GCV. \par 
The following two lemmas state that the cross components of the prediction risk and GCV are asymptotically negligible.
\begin{lem}
\label{lem1}
Suppose that Assumptions 1--4 hold. Then, for all $\lambda \in \Lambda$, we have 
\begin{equation*}
    R_c(\lambda)\xrightarrow{a.s.}0.
\end{equation*}
\end{lem}

\begin{lem}
\label{lem2}
Suppose that Assumptions 1--4 hold. Then, for all $\lambda \in \Lambda$, we have
\begin{equation*}
    \frac{\gcv_c(\lambda)}{\gcv_d(\lambda)}\xrightarrow{a.s.}0.
\end{equation*}
\end{lem}
Note that these cross components are more complex than those in \cite{pwrt2021} owing to the misspecification of the regression function. \par 
The following lemma states that the bias components of GCV and the prediction risk are asymptotically equivalent.   
\begin{lem}
\label{lem3}
Suppose that Assumptions 1--4 hold. Then, for all $\lambda \in \Lambda$, we have
\begin{equation*}
    R_b(\lambda)-\frac{\gcv_b(\lambda)}{\gcv_d(\lambda)}\xrightarrow{a.s.}0.
\end{equation*}
\end{lem}
A significant difference between our study and \cite{pwrt2021} arises in the process of showing the asymptotic equivalence of the bias components. Because the regression function is correctly specified as a linear function in \cite{pwrt2021}, their bias components are relatively simple. On the other hand, because the regression function is incorrectly specified as a linear function in this study, the bias components are more complex owing to the existence of the specification error. However, Assumption 4 makes these components tractable. \par 

The variance components of GCV and the prediction risk are the same as those in \cite{pwrt2021}. They showed that these variance components are asymptotically equivalent regardless of the model specification, which is stated in the following lemma. 
\begin{lem}
\label{lem4}
\textup{(Lemma 5.4 in \citealt{pwrt2021}.)} Suppose that Assumptions 1--3 hold. Then, for all $\lambda \in \Lambda$, we have
\begin{equation*}
    R_v(\lambda)-\frac{\gcv_v(\lambda)}{\gcv_d(\lambda)}\xrightarrow{a.s.}0.
\end{equation*}
\end{lem}
By employing Lemmas 1--4, we obtain the following theorem, which states that GCV is uniformly consistent under the misspecification.
\begin{thm}
\label{thm1}
Suppose that Assumptions 1--4 hold. Then, we have
\begin{equation*}
    \sup_{\lambda \in \Lambda}|\gcv(\lambda)-R(\lambda)|\xrightarrow{a.s.}0. 
\end{equation*}
  
\end{thm}
Lemmas 1--4 imply the pointwise convergence of GCV. However, the pointwise convergence is not sufficient for the asymptotic optimality of GCV because $\lambdahat$ and $\lambda^*$ are random variables and depend on $n$, and we require the uniform convergence. To obtain Theorem 1, we show that the difference between GCV and the prediction risk is equicontinuous on $\Lambda$, which makes the Arzela-Ascoli theorem hold. As in \cite{pwrt2021}, we derive the equicontinuity by showing that the prediction risk, GCV, and their derivatives are bounded on $\Lambda$. \par 

Theorem 1 implies that GCV tuning is asymptotically optimal tuning. This is because $R(\lambdahat)\geq R(\lambda^*)$ by the definition of $\lambda^*$, and we have 
\begin{align*}
0&\leq R(\lambdahat)-R(\lambda^*)=R(\lambdahat)-\gcv(\lambdahat)+\gcv(\lambdahat)-\gcv(\lambda^*)-(R(\lambda^*) -\gcv(\lambda^*))\\
&\leq R(\lambdahat)-\gcv(\lambdahat)-(R(\lambda^*)-\gcv(\lambda^*)) \leq 2 \sup_{\lambda \in \Lambda}|R(\lambda)-\gcv(\lambda)| \xrightarrow{a.s}0.
\end{align*}
\par 

\begin{rem}
\textup{
Although in this study we investigated GCV only for the high-dimensional misspecified linear model, the uniform consistency also holds for LOOCV under the same assumption if $\expect[f(x_i)]=0$. This is because \cite{pwrt2021} showed that GCV and LOOCV are asymptotically equivalent regardless of the model specification for mean zero response variables with the bounded second moment. See Theorem 4.2 of \cite{pwrt2021}. } 
\end{rem}

\section{Simulation Studies}
In this section, we investigate the prediction performance of ridge regressions that are tuned by GCV and LOOCV. For reference, we also compare it with the prediction performance of ridge regression with optimal tuning, the min-norm least squares, and the Lasso. To implement the Lasso, we use the R package \verb|glmnet| (\citealt{fht2010}). We choose the tuning parameter for the Lasso using 10-fold cross-validation with the \verb|lambda.min| option. The candidates of the tuning parameter for the Lasso are also generated by \verb|glmnet|. The set of candidates for the tuning parameter for ridge regression is given by $\Lambda$, including 100 breakpoints of the interval $[\lambda_{\min},\lambda_{\max}]$ and zero, where $\lambda_{\min}=-(1-\sqrt{\gamma})^2/5$ and $\lambda_{\max}=5$. In all settings, we set the covariance of the covariates as $\Sigma=I_p+\rho\boldsymbol{1}\boldsymbol{1}^T$, where $\rho=0.1$ and $\boldsymbol{1}=(1,\ldots,1)^T \in \real^{p}$, and this implies $\lambda_{\min}(\Sigma)=1$. We also set $n=100$ and $p=\gamma n$, where $\gamma$ is given by
\begin{equation*}
    \gamma \in \{0.1,0.2,0.5,0.9,1.1,2,5,10,20,30,40\}.
\end{equation*}
\par 
We consider the following data generating process: 
\begin{equation*}
    y_i=f(x_i)+\epsilon_i,
\end{equation*}
where $x_i \sim N(0, \Sigma)$ and $\epsilon_i|x_i \sim  N(0,1)$. $f(x_i)$ takes one of the following forms: 
\begin{align}
    f(x_i)&=x_i^T\beta, \label{263} \\ 
    f(x_i)&=x_i^T\beta+\frac{1}{\sqrt{(1+\rho)p \log n}}\|x_i\|, \label{264}\\ 
    f(x_i)&=x_i^T\beta+\frac{1}{\log p}\max_{1\leq j \leq p}|x_{ij}|, \label{265}
\end{align}
where $\beta \in \mathbb{R}^p$ is given by $\beta=(b,\ldots,b)$ with $b=\sqrt{\alpha/(p+p^2\rho)}$, and $\alpha=10$. If $f(x_i)$ is defined as in \eqref{263}, the data generating process is the same as in \cite{kls2020}. If $f(x_i)$ is defined as in \eqref{264}, simple calculations yield
\begin{equation*}
    \expect(\delta_i^2)\leq \frac{1}{(1+\rho)p\log n}\expect(\|x_i\|^2)=\frac{1}{\log n}.
\end{equation*}
If $f(x_i)$ is defined as in \eqref{265}, then we have 
\begin{equation*}
    \expect(\delta_i^2)\leq\frac{1}{(\log p)^2}\expect\left(\max_{1 \leq j\leq p}|x_{ij}|^2\right)\leq \frac{4(1+\rho)\log(\sqrt{2}p)}{(\log p)^2}=O\left(\frac{1}{\log p}\right). 
\end{equation*}
These specification errors satisfy Assumption 4. 
\par  

To compare the prediction accuracy of the four methods, we calculate the averaged empirical prediction risks. Let $\betahat$ be an estimator of the linear projection coefficient $\beta_0$. The empirical prediction risk of $\betahat$ at the $k$th repetition is defined as
\begin{equation*}
    \text{EPR}(\betahat^{(k)})=\frac{1}{L}\sum_{\ell=1}^L\left(y_{0\ell}^{(k)}-x_{0\ell}^{(k)T}\betahat^{(k)}\right)^2,
\end{equation*}
where $L$ is the size of the test data, which we set as $L=1000$, $(y_{0\ell}^{(k)},x_{0\ell}^{(k)})$ are the test data at the $k$th repetition for $\ell=1,\ldots, L$, and $\betahat^{(k)}$ is the realization of $\betahat$ calculated by the training data at the $k$th repetition. Note that the optimal tuning parameter for ridge regression at the $k$th repetition is the value in $\Lambda$ that minimizes the empirical prediction risk at the $k$th repetition. Then, the averaged empirical prediction risk of $\betahat$ is given by
\begin{equation*}
    \text{AEPR}(\betahat)=\frac{1}{K}\sum_{k=1}^K\text{EPR}(\betahat^{(k)}),
\end{equation*}
where $K$ is the repetition number of experiments, and we set $K=100$. 

Next, we present the simulation results. Table \ref{table1} shows the averaged empirical prediction risks of ridge regressions, the min-norm least squares, and the Lasso when $f(x_i)$ is given by \eqref{263}. When $\gamma \leq 0.2$, the empirical prediction risks of these four methods are of similar values and they increase as $\gamma$ approaches 1. In particular, the empirical prediction risk of the min-norm least squares becomes inflated owing to overfitting. However, when $\gamma$ exceeds 1, the empirical prediction risks of ridge regressions and the min-norm least squares begin to decrease, and these trends continue as $\gamma$ increases. They indicate that the large number of covariates enhances the prediction performance of ridge regression and the min-norm least squares. However, the empirical prediction risk of the Lasso continues to increase when $\gamma$ exceeds 1, indicating that including many covariates does not necessarily improve the prediction performance of the Lasso. The prediction performance of ridge regressions tuned by GCV and LOOCV is generally close to that of optimally tuned ridge regression. When $\gamma<10$, ridge regressions tuned by GCV and LOOCV considerably outperform the min-norm least squares, but the performance of the three methods is similar when $\gamma \geq 10$. The prediction performance of the Lasso is worse than that of ridge regressions tuned by GCV and LOOCV for all $\gamma$.  \par

Table \ref{table2} shows the averages of selected tuning parameters for ridge regressions when $f(x_i)$ is given by \eqref{263}. GCV and LOOCV select similar values on average, and they are close to the averaged optimal tuning parameters in general. The values selected by GCV and LOOCV are increasing when $\gamma \leq 5$ but begin to decrease when $\gamma >5$, and they become negative when $\gamma\geq 20$.  \par

Table \ref{table3} presents the averaged empirical prediction risks of ridge regressions, the min-norm least squares, and the Lasso when $f(x_i)$ is given by \eqref{264}. Owing to the misspecification of the regression function, the empirical prediction risks of all methods become larger for all $\gamma$. However, their behaviors are similar to the case of the correct specification. The empirical prediction risks of ridge regressions are increasing when $\gamma \leq 1$, and they are decreasing when $\gamma>1$. The performances of GCV and LOOCV are also similar to that of optimal tuning under this misspecification. On the other hand, the empirical prediction risk of the Lasso is increasing, and it is larger than that of ridge regressions tuned by GCV and LOOCV for all $\gamma$. \par 

Table \ref{table4} presents the averages of the selected tuning parameters for ridge regressions when $f(x_i)$ is given by \eqref{264}. GCV and LOOCV also selected values close to the optimal tuning parameters under this misspecification for all $\gamma$. The selected tuning parameters for ridge regressions become somewhat larger than the correct specification. They are increasing when $\gamma \leq 5$ and decreasing when $\gamma \geq 10$. 

\par

Table \ref{table5} shows the averaged empirical prediction risks of ridge regressions, the min-norm least squares, and the Lasso when $f(x_i)$ is given by \eqref{265}. All empirical prediction risks are relatively large when $p$ is small because the specification errors are large in that case. The empirical prediction risks of ridge regressions and the min-norm least squares reach their minimums at $\gamma=40$. By contrast, the empirical prediction risk of Lasso reaches the minimum when $\gamma=0.2$ and the maximum when $\gamma=40$. Ridge regressions tuned by GCV and LOOCV also exhibit similar prediction performance to optimally tuned ridge regressions under this misspecification. When $\gamma \leq 5$, ridge regressions considerably outperform the min-norm least squares, but when $\gamma>5$, their performances are similar. Ridge regressions tuned by GCV and LOOCV outperform the Lasso for all $\gamma$. \par 

Table \ref{table6} shows the averages of the selected tuning parameters for ridge regressions when $f(x_i)$ is given by \eqref{265}. The tuning parameters selected by GCV and LOOCV are also close to the optimally selected tuning parameters when $\gamma \leq 10$. When $\gamma=20$, the averaged optimal tuning parameter for ridge regression becomes positive, but the averaged selected values by GCV and LOOCV become negative. However, the prediction performance of these methods at these tuning parameters is almost the same. The optimal tuning parameters and selected values by GCV and LOOCV are negatively large when $\gamma\geq 30$. \par

Note that we also performed simulation studies where $f(x_i)$ is given by \eqref{197}. Because the results are similar to those in Tables \ref{table1} and \ref{table2}, we do not report them here.

In summary, our simulation studies show that the prediction performance of ridge regressions tuned by GCV and LOOCV is always close to that of optimally tuned ridge regression under correct and incorrect model specifications. Almost no differences exist between the performances of GCV and LOOCV. We observed that ridge regressions tuned by GCV and LOOCV outperform the min-norm least squares when $n$ is larger than $p$ or when $p$ is moderately larger than $n$, but their performance is almost the same when $p$ is much larger than $n$. We also observed that the ridge regressions tuned by GCV and LOOCV always outperform the Lasso regardless of the model specification or the ratio of $p$ and $n$ under the considered settings.

\begin{table}
\centering
\caption{Averaged empirical prediction risks when $f(x_i)$ is given by \eqref{263}, $p=\gamma n$, and $n=100$.}
\hspace*{-2.5em}
\begin{tabular}{cccccccccccc}
\midrule
$\gamma$   & 0.1   & 0.2   & 0.5   & 0.9    & 1.1    & 2     & 5     & 10    & 20    & 30 & 40\\ \midrule
Optimal ridge & 1.098 & 1.188 & 1.348 & 1.395  & 1.394  & 1.340   & 1.248 & 1.182 & 1.136  &1.116  & 1.116\\
GCV ridge     & 1.118 & 1.206 & 1.367 & 1.414  & 1.528  & 1.364  & 1.267  & 1.192 & 1.142  &1.120  &1.118 \\
LOOCV ridge      & 1.118 & 1.206 & 1.366 & 1.414  & 1.496  & 1.365  & 1.267 & 1.193 & 1.141   &1.119  & 1.118\\
Min-norm LS & 1.122 & 1.249 & 1.999 & 11.666 & 11.046 & 2.091 & 1.333 & 1.193 & 1.141   &1.124  & 1.129 \\
Lasso   & 1.121 & 1.251 & 1.789 & 2.240   & 2.275  & 2.417  & 2.472 & 2.678 & 2.879  &2.980  & 3.089\\ \midrule
\end{tabular}
\label{table1}

\vspace{2em}

\caption{Averages of selected tuning parameters for ridge regressions when $f(x_i)$ is given by \eqref{263}, $p=\gamma n$, and $n=100$.}
 \hspace*{-2.5em}
\begin{tabular}{cccccccccccc}
\midrule
$\gamma$   & 0.1   & 0.2   & 0.5   & 0.9   & 1.1   & 2     & 5     & 10    & 20    & 30    & 40\\ \midrule
Optimal ridge & 0.048 & 0.137 & 0.471  & 0.836 & 0.948 & 1.436 & 1.896 & 1.503 & -0.770  & -2.531& -4.984 \\
GCV ridge     & 0.038 & 0.129 & 0.488 & 0.842 & 0.920 & 1.376  & 1.708 & 1.327 & -0.721  &-2.370 & -4.527\\
LOOCV ridge      & 0.035 & 0.130 & 0.476 & 0.829 & 0.905 & 1.327 & 1.644  & 1.144 & -0.885  &-2.542 & -4.712\\
 \midrule
\end{tabular}
\label{table2}

\vspace{2em}

\caption{Averaged empirical prediction risks when $f(x_i)$ is given by \eqref{264}, $p=\gamma n$, and $n=100$.}
 \hspace*{-2.5em}
\begin{tabular}{cccccccccccc}
\midrule
$\gamma$    & 0.1   & 0.2   & 0.5   & 0.9    & 1.1    & 2     & 5     & 10    & 20  & 30  & 40\\ \midrule
Optimal ridge  & 1.347 & 1.434 & 1.612 & 1.646  & 1.645  & 1.592 & 1.490  & 1.424 & 1.365 & 1.347 &1.342 \\
GCV ridge      & 1.367 & 1.455 & 1.640  & 1.668  & 1.851  & 1.641  & 1.516 & 1.438 & 1.373 & 1.352 & 1.345 \\
LOOCV ridge       & 1.369 & 1.458 & 1.638 & 1.672  & 1.725  & 1.641 & 1.514 & 1.438  & 1.372  & 1.352 & 1.345\\
Min-norm LS & 1.377 & 1.518 & 2.472 & 14.488 & 13.847 & 2.523 & 1.608 & 1.443 & 1.370   & 1.354 &1.353\\
Lasso    & 1.376 & 1.519 & 2.152 & 2.542  & 2.608  & 2.714 & 2.670  & 2.850  & 3.051 & 3.15  & 3.254\\ \midrule
\end{tabular}
\label{table3}

\vspace{2em}

\caption{Averages of selected tuning parameters for ridge regressions when $f(x_i)$ is given by \eqref{264}, $p=\gamma n$, and $n=100$.}
 \hspace*{-2.5em}
\begin{tabular}{cccccccccccc}
\midrule
$\gamma$   & 0.1   & 0.2   & 0.5   & 0.9   & 1.1   & 2     & 5     & 10    & 20   &30  & 40\\ \midrule
Optimal ridge & 0.063 & 0.161 & 0.537 & 0.938 & 1.050 & 1.592 & 2.231 & 2.021 & -0.065  &-1.802 &-4.507 \\
GCV ridge     & 0.051  & 0.152 & 0.536 & 0.926 & 1.029 & 1.495 & 1.940   & 1.913 & -0.157 & -1.588 &-3.774 \\
LOOCV ridge      & 0.048 & 0.153 & 0.531 & 0.910 & 1.028 & 1.448 & 1.883  & 1.758 & -0.357  & -1.818& -4.015\\
\midrule
\end{tabular}
\label{table4}
\end{table}

\begin{table}
\centering
 \caption{Averaged empirical prediction risks when $f(x_i)$ is given by \eqref{265}, $p=\gamma n$, and $n=100$.}
 \hspace*{-2.5em}
\begin{tabular}{cccccccccccc}

\midrule
$\gamma$    & 0.1   & 0.2   & 0.5   & 0.9    & 1.1    & 2     & 5     & 10    & 20   & 30  & 40\\ \midrule
Optimal ridge  & 1.974 & 1.867 & 1.901 & 1.860  & 1.839  & 1.737 & 1.579  & 1.480 & 1.396 & 1.365 &  1.351 \\
GCV ridge      & 2.001 & 1.897 & 1.935  & 1.888  & 2.005  & 1.795  & 1.605 & 1.494 & 1.403 & 1.37 & 1.354\\
LOOCV ridge       & 2.002 & 1.897 & 1.935 & 1.888  & 1.927  & 1.788 & 1.604 & 1.495  & 1.403 & 1.37& 1.354\\
Min-norm LS & 2.026 & 1.992 & 2.971 & 16.580 & 15.899 & 2.772 & 1.709 & 1.501 & 1.400   & 1.372& 1.362\\
Lasso    & 2.025 & 1.995 & 2.528 & 2.802  & 2.858  & 2.899 & 2.777  & 2.889  & 3.067 & 3.162& 3.258\\ \midrule
\end{tabular}
\label{table5}

\vspace{2em}

\caption{Averages of selected tuning parameters for ridge regressions when $f(x_i)$ is given by \eqref{265}, $p=\gamma n$, and $n=100$.}
\hspace*{-2.5em}
\begin{tabular}{cccccccccccc}
\midrule
$\gamma$   & 0.1   & 0.2   & 0.5   & 0.9   & 1.1   & 2     & 5     & 10    & 20   &30  & 40\\ \midrule
Optimal ridge & 0.087 & 0.198 & 0.591 & 1.007 & 1.126 & 1.674 & 2.340 & 2.134 & 0.021 & -1.731 & -4.508\\
GCV ridge     & 0.064  & 0.186 & 0.591 & 0.986 & 1.111 & 1.557 & 2.036   & 2.037 & -0.067 &-1.549& -3.755\\
LOOCV ridge      & 0.065 & 0.192 & 0.585 & 0.974 & 1.103 & 1.521 & 1.981  & 1.888 & -0.272 &-1.788& -3.995\\
\midrule
\end{tabular}
\label{table6}  
\end{table}


\section{Conclusion}
In this study, we demonstrated that generalized cross-validation can select the asymptotically optimal tuning parameter for ridge regression in high-dimensional misspecified linear models if the second moment of the specification error converges to zero. We permitted the candidates for the tuning parameter for ridge regression to be not only positive values but also zero and negative values. Furthermore, our simulation studies showed that ridge regression tuned by generalized cross-validation performed like optimally tuned ridge regression and outperformed the Lasso under correct and incorrect model specifications. Our results suggest that generalized cross-validation remains an effective method to tune high-dimensional ridge regression for prediction if the degree of model misspecification is small.

\section*{Acknowledgments}
The author would like to thank Naoya Sueishi for his valuable comments. This study was supported by JST SPRING (Grant Number JPMJFS2126) and the Institute for Economic and Business Research, Shiga University.

\appendix
\section*{Appendix}
We add some notations. Let $\|A\|$ be the operator norm of a matrix $A$. Let $C$ be a generic positive constant that can vary from line to line.\par 
We decompose the prediction risk $R(\lambda)$ into the bias component $R_b(\lambda)$, the variance component $R_v(\lambda)$, and the cross component $R_c(\lambda)$. Under Assumption 4, we can write
\begin{align*}
R_b(\lambda)&=\expect(\delta_0^2)+\beta_0^T\{\Sigmahat(\Sigmahat+\lambda I_p)^+-I_p\}\Sigma\{\Sigmahat(\Sigmahat+\lambda I_p)^+-I_p\}\beta_0\\
&\ \ +\frac{2}{n} \delta^TX(\Sigmahat+\lambda I_p)^+\Sigma\{\Sigmahat(\Sigmahat+\lambda I_p)^+-I_p\}\beta_0 \\
&\ \ +\fn \delta^TX(\Sigmahat +\lambda I_p)^+\Sigma (\Sigmahat +\lambda I_p)^+\fn X^T\delta,\\  
    R_v(\lambda)&=\sigma^2+\fn \epsilon^TX(\Sigmahat+\lambda I_p)^+\Sigma (\Sigmahat +\lambda I_p)^+\fn X^T\epsilon, \\
    R_c(\lambda)&=\frac{2}{n}\delta^TX(\Sigmahat+\lambda I_p)^+\Sigma(\Sigmahat+\lambda I_p)^+\fn X^T\epsilon \\
    &\ \ +\frac{2}{n}\epsilon^TX(\Sigmahat+\lambda I_p)^+\Sigma \{\Sigmahat(\Sigmahat+\lambda I_p)^+-I_p\}\beta_0,
\end{align*}
where $\delta_0=f(x_0)-x_0^T\beta_0$. Simple calculations yield 
\begin{equation*}
R(\lambda)=R_b(\lambda)+R_v(\lambda)+R_c(\lambda).  
\end{equation*}

Similarly, we decompose the numerator of $\gcv(\lambda)$ into the bias component $\gcv_b(\lambda)$, the variance component $\gcv_v(\lambda)$, and the cross component $\gcv_c(\lambda)$: 
\begin{align*}
    \gcv_b(\lambda)&=\fn \beta_0^TX^T(I_n-S_\lambda )^2X\beta_0+\frac{2}{n}\beta_0^TX^T(I_n-S_\lambda)^2\delta+\fn \delta^T(I_n-S_\lambda)^2\delta,\\
    \gcv_v(\lambda)&=\fn \epsilon^T(I_n-S_\lambda)^2\epsilon,  \\
    \gcv_c(\lambda)&=\frac{2}{n}\beta_0^TX^T(I_n-S_\lambda)^2\epsilon+\frac{2}{n}\delta^T(I_n-S_\lambda)^2 \epsilon.
\end{align*}
The denominator of $\gcv_d(\lambda)$ can be written as
\begin{equation*}
     \gcv_d(\lambda)=\left(1-\tr(S_\lambda)\fn\right)^2=\left(1-\tr[(\Sigmahat +\lambda I_p)^+\Sigmahat]\fn\right)^2,
\end{equation*}
and simple calculations yield
\begin{equation*}
    \gcv(\lambda)=\frac{\gcv_b(\lambda)}{\gcv_d(\lambda)}+\frac{\gcv_v(\lambda)}{\gcv_d(\lambda)}+\frac{\gcv_c(\lambda)}{\gcv_d(\lambda)}.
\end{equation*}
Recall that we are defining $\gcv (0):=\lim_{\lambda \rightarrow 0}\gcv(\lambda)$.

\begin{proof}[Proof of Lemma \ref{lem1}]
We begin by evaluating the first term of $R_c(\lambda)$. The Cauchy--Schwarz inequality yields
\begin{align*}
\frac{2}{n}\left|\delta^TX(\Sigmahat +\lambda I_p)^+\Sigma (\Sigmahat +\lambda I_p)^+\fn X^T\epsilon \right|
    &\leq \frac{2}{n^2}\|\delta\|\|X(\Sigmahat +\lambda I_p)^+\Sigma (\Sigmahat +\lambda I_p)^+X^T\epsilon \|\\
    &\leq \frac{2}{n^2} \lambda_{\max}(\Sigma)\|\delta\| \|X(\Sigmahat+\lambda I_p)^+\|^2  \|\epsilon\|=(i), 
\end{align*}
for all $\lambda \in \Lambda$. Note that $\lambda_{\min}(\fn XX^T)\geq \lambda_{\min}(\Sigma)\lambda_{\min}(\fn ZZ^T) \geq  C_{\min} \lambda_{\min}(\fn ZZ^T)$, where $Z\in \mathbb{R}^{n \times p}$ is the matrix whose $i$th row is $z_i$. The Bai--Yin theorem (Theorem 1 in \citealt{by1993}) yields
\begin{equation*}
    \lambda_{\min}\left(\frac{1}{p}ZZ^T\right)\xrightarrow{a.s.}\left(1-\frac{1}{\sqrt{\gamma}}\right)^2.
\end{equation*}
Thus, we obtain 
\begin{equation*}
    \lambda_{\min}\left(\frac{1}{n}ZZ^T\right)=\frac{p}{n} \lambda_{\min}\left(\frac{1}{p}ZZ^T\right)\xrightarrow{a.s.}\gamma \left(1-\frac{1}{\sqrt{\gamma}}\right)^2=(1-\sqrt{\gamma})^2.
\end{equation*}
Therefore, for sufficiently large $n$, we almost surely obtain
\begin{align*}
    \lambda_{\min}\left(\frac{1}{n}XX^T\right)+\lambda_{\min}&> \kappa C_{\min}(1-\sqrt{\gamma})^2 -\kappa^2 C_{\min}(1-\sqrt{\gamma})^2\\
    &=\kappa(1-\kappa) C_{\min}(1-\sqrt{\gamma})^2>0.
\end{align*}
Similar arguments show that for sufficiently large $n$ and some finite positive constant $K>1$, 
\begin{equation*}
    \lambda_{\max}\left(\fn X^TX\right)=\lambda_{\max}\left(\fn XX^T\right)< K C_{\max}(1-\sqrt{\gamma})^2,
\end{equation*}
almost surely. Note that
\begin{align*}
    \left\|\frac{1}{\sqrt{n}}X(\Sigmahat+\lambda I_p)^+\right\|^2&=\lambda_{\max}\left(\left(\frac{1}{\sqrt{n}}X(\Sigmahat+\lambda I_p)^+\right)^T\left(\frac{1}{\sqrt{n}}X(\Sigmahat+\lambda I_p)^+\right)\right)\\
    &=\lambda_{\max}\left((\Sigmahat+\lambda I_p)^+\Sigmahat(\Sigmahat+\lambda I_p)^+\right),
\end{align*}
for all $\lambda \in \Lambda$. Let 
\begin{equation}
\Sigmahat=V\diag(\lambda_1,\ldots,\lambda_n,0,\ldots,0)V^T  \label{2286}
\end{equation}
be the eigen decomposition of $\Sigmahat$, where $V \in \mathbb{R}^{p\times p}$ is an orthogonal matrix, and $\lambda_1\geq \cdots\geq \lambda_n$ are the decreasing ordered eigenvalues of $\Sigmahat$. Note that $\lambda_1=\lambda_{\max}(\Sigmahat)$, and $\lambda_n=\lambda_{\min}(\fn XX^T)$. Moreover, $\Sigmahat^+$ can be expressed as $\Sigmahat^+=V\text{diag}(\lambda_1^{-1},\ldots,\lambda_{n}^{-1},0,\ldots,0)V^T.$ We can now write 
\begin{equation*}
    (\Sigmahat+\lambda I_p)^+\Sigmahat (\Sigmahat+\lambda I_p)^+=V\diag\left(\frac{\lambda_1}{(\lambda_1+\lambda)^2},\ldots, \frac{\lambda_n}{(\lambda_n+\lambda)^2},0,\ldots,0\right)V^T.
\end{equation*}
Note that for $i=1,\ldots,n$, sufficiently large $n$, and all $\lambda \in \Lambda$, we have
\begin{equation*}
    \frac{\lambda_i}{(\lambda_i+\lambda)^2}\leq \frac{\lambda_1}{(\lambda_n+\lambda_{\min})^2}< \frac{KC_{\max}}{\kappa^2(1-\kappa)^2 C_{\min}^2(1-\sqrt{\gamma})^2},
\end{equation*}
almost surely. Therefore, for sufficiently large $n$ and all $\lambda \in \Lambda$, we almost surely obtain
\begin{equation*}
  \left\|\frac{1}{\sqrt{n}}X(\Sigmahat+\lambda I_p)^+\right\|^2=\lambda_{\max}\left((\Sigmahat+\lambda I_p)^+\Sigmahat(\Sigmahat+\lambda I_p)^+\right)< \frac{KC_{\max}}{\kappa^2(1-\kappa)^2 C_{\min}^2(1-\sqrt{\gamma})^2}.
\end{equation*}
Therefore, for sufficiently large $n$ and all $\lambda \in \Lambda$, we almost surely obtain
\begin{equation*}
    (i)< \frac{2KC_{\max}^2}{\kappa^2(1-\kappa)^2 C_{\min}^2(1-\sqrt{\gamma})^2}  \frac{1}{\sqrt{n}}\|\delta\| \frac{1}{\sqrt{n}}\|\epsilon \|\leq \frac{C}{\sqrt{n}}\|\delta\|.
\end{equation*}
This implies that the first term of $R_c(\lambda)$ almost surely converges to zero for all $\lambda \in \Lambda$. Furthermore, \cite{pwrt2021} showed that for all $\lambda \in \Lambda$,\begin{equation*}
    \frac{2}{n}\epsilon^TX(\Sigmahat+\lambda I_p)^+\Sigma \{\Sigmahat(\Sigmahat+\lambda I_p)^+-I_p\}\beta_0\xrightarrow{a.s.}0.
\end{equation*}
See the proof of their Lemma 5.1. Therefore, we obtain for all $\lambda \in \Lambda$,
\begin{equation*}
    R_c(\lambda)\xrightarrow{a.s.}0.
\end{equation*}
This completes the proof of this lemma. \par
\end{proof}

\begin{proof}[Proof of Lemma \ref{lem2}]
Because we are defining $\gcv(0):=\lim_{\lambda \rightarrow 0}\gcv(\lambda)$, we can write 
\begin{equation}
    \frac{\gcv_c(\lambda)}{\gcv_d(\lambda)}=\frac{\frac{2}{n}\beta_0^T(\Sigmahat+\lambda I_p)^{+2} X^T\epsilon}{(\fn \tr[(\fn XX^T+\lambda I_n)^+])^2}+\frac{\frac{2}{n}\delta^T(\fn XX^T+\lambda I_n)^{+2}\epsilon}{(\fn \tr[(\fn XX^T+\lambda I_n)^+])^2}, \label{321}
\end{equation}
for all $\lambda \in \Lambda$. \cite{pwrt2021} showed that the first term of $\gcv_c(\lambda)/\gcv_d(\lambda)$ almost surely converges to zero for all $\lambda \in \Lambda$. See the proof of their Lemma 5.2. The second term of $\gcv_c(\lambda)/\gcv_d(\lambda)$ is bounded as follows:
\begin{align*}
   &\frac{2}{n}\left|\frac{\delta^T(\fn XX^T+\lambda I_n)^{+2}\epsilon}{(\fn \tr[(\fn XX^T+\lambda I_n)^+])^2 }\right|
    \leq \frac{2}{n}\left(\lambda_{\max}\left(\fn XX^T\right)+\lambda  \right)^2\left|\delta^T\left(\fn XX^T+\lambda I_n\right)^{+2}\epsilon \right|\\
   &\leq \frac{2}{n}\left(\lambda_{\max}\left(\fn XX^T\right)+\lambda  \right)^2\left\|\left(\fn XX^T+\lambda I_n\right)^{+2}\epsilon\right\|\|\delta\|=(ii),
\end{align*}
for all $\lambda \in \ \Lambda$. For sufficiently large $n$ and all $\lambda \in \Lambda$, we almost surely obtain
\begin{align*}
    \fn \left\|\left(\fn XX^T+\lambda I_n\right)^{+2}\epsilon\right\|^2&=\fn \epsilon ^T\left(\fn XX^T+\lambda I_n\right)^{+4}\epsilon\\
    &< \frac{1}{\kappa^4(1-\kappa)^4 C_{\min}^4(1-\sqrt{\gamma})^8}\fn \|\epsilon\|^2.
\end{align*}
Furthermore, we almost surely obtain $\lambda_{\max}(\fn XX^T)+\lambda<KC_{\max}(1-\sqrt{\gamma})^2+\lambda_{\max}$ for sufficiently large $n$ and all $\lambda \in \Lambda$. Therefore, for sufficiently large $n$ and all $\lambda \in \Lambda$, we almost surely obtain
\begin{equation*}
    (ii)< \frac{2(KC_{\max}(1-\sqrt{\gamma}^2)+\lambda_{\max})^2}{\kappa^2(1-\kappa)^2 C_{\min}^2(1-\sqrt{\gamma})^4}\frac{1}{\sqrt{n}} \|\epsilon\| \frac{1}{\sqrt{n}}\|\delta\| \leq \frac{C}{\sqrt{n}}\|\delta\|.
\end{equation*}
This implies that the second term of $\gcv_c(\lambda)/\gcv_d(\lambda)$ almost surely coverages to zero for all $\lambda \in \Lambda$. Therefore, we obtain
\begin{equation*}
    \frac{\gcv_c(\lambda)}{\gcv_d(\lambda)}\xrightarrow{a.s.}0,
\end{equation*}
for all $\lambda \in \Lambda$. This completes the proof of this lemma.

\end{proof}

\begin{proof}[Proof of Lemma \ref{lem3}]
Owing to the definition of $\gcv(0):=\lim_{\lambda \rightarrow 0}\gcv(\lambda)$, we can write
\begin{equation*}
    R_b(\lambda)-\frac{\gcv_b(\lambda)}{\gcv_d(\lambda)}=a_\lambda+b_\lambda+c_\lambda,
\end{equation*}
where \begin{align*}
   a_\lambda&=\beta_0^T\{\Sigmahat(\Sigmahat+\lambda I_p)^{+}-I_p\}\Sigma\{(\Sigmahat+\lambda I_p)^{+}\Sigmahat-I_p\}\beta_0-\frac{\fn \beta_0^TX^T(\fn XX^T+\lambda I_n)^{+2}X\beta_0}{(\fn \tr[(\fn XX^T+\lambda I_n)^+] )^2}, \\
     b_\lambda&=\expect(\delta_0^2)+\fn \delta^TX(\Sigmahat+\lambda I_p)^{+}\Sigma(\Sigmahat+\lambda I_p)^{+}\fn X^T\delta-\frac{\fn \delta^T(\fn XX^T+\lambda I_n)^{+2}\delta}{(\fn \tr[(\fn XX^T+\lambda I_n)^+] )^2},\\
     c_\lambda&=2 \beta_0^T\{\Sigmahat(\Sigmahat+\lambda I_p)^{+}-I_p\}\Sigma(\Sigmahat+\lambda I_p)^{+}\fn X^T\delta-\frac{\frac{2}{n}\delta^T(\fn XX^T+\lambda I_n)^{+2}X\beta_0}{(\fn \tr[(\fn XX^T+\lambda I_n)^+])^2},
\end{align*}
for all $\lambda \in \Lambda$. \cite{pwrt2021} showed that $a_\lambda$ almost surely converges to zero for all $\lambda \in \Lambda$. See the proof of their Lemma 5.3. We first evaluate $b_\lambda$. $\expect(\delta_0^2)=o(1)$ by Assumption 4. The second term of $b_\lambda$ is bounded as follows:
\begin{align*}
 \fn \delta^TX (\Sigmahat +\lambda I_p)^+\Sigma (\Sigmahat +\lambda I_p)^+\fn X^T\delta &\leq \lambda_{\max}(\Sigma)  \fn \delta^TX(\Sigmahat+\lambda I_p)^{+2}X^T\fn \delta\\
  &\leq C_{\max} \lambda_{\max}\left(\fn X(\Sigmahat+\lambda I_p)^{+2}X^T\right) \fn \|\delta\|^2=(iii),
\end{align*}
for all $\lambda \in \Lambda$. Let 
\begin{equation*}
    \frac{1}{\sqrt{n}}X=UQ_0V^T
\end{equation*}
be a singular value decomposition of $X/\sqrt{n}$, where $U \in \mathbb{R}^{n\times n}$ and $V \in \mathbb{R}^{p \times p}$ are orthogonal matrices, and $Q_0\in \mathbb{R}^{n \times p}$ is the matrix whose $j$th diagonal element is $\sqrt{\lambda_j}$ for $j=1,\ldots,n$, and all non-diagonal elements are zero. Recall that $\lambda_j$, $j=1,\ldots,n$ are eigenvalues of $\Sigmahat$, and $V$ is the same as in \eqref{2286}. We can now write
\begin{equation*}
\fn X(\Sigmahat+\lambda I_p)^{+2}X^T=U\diag\left(\frac{\lambda_1}{(\lambda_1+\lambda)^2},\ldots,\frac{\lambda_n}{(\lambda_n+\lambda)^2}\right)U^T,
\end{equation*}
for all $\lambda \in \Lambda$. Thus, for sufficiently large $n$ and all $\lambda \in \Lambda$, we almost surely obtain  
\begin{equation*}
\lambda_{\max}\left(\fn X(\Sigmahat+\lambda I_p)^{+2}X^T\right)<\frac{KC_{\max}}{\kappa^2(1-\kappa)^2 C_{\min}^2(1-\sqrt{\gamma})^2}.
\end{equation*}
Therefore, for sufficiently large $n$ and all $\lambda \in \Lambda$, we almost surely obtain
\begin{equation*}
    (iii)<\frac{KC_{\max}^2}{\kappa^2(1-\kappa)^2 C_{\min}^2(1-\sqrt{\gamma})^2}\fn \|\delta\|^2.
\end{equation*}
This implies that the second term of $b_\lambda$ almost surely converges to zero for all $\lambda \in \Lambda$. The third term of $b_\lambda$ is almost surely bounded as follows:
\begin{equation*}
  \fn \frac{\delta^T\left(\fn XX^T+\lambda I_n\right)^{+2}\delta}{ \left(\fn\tr\left[ (\fn XX^T+\lambda I_n)^{+}\right]\right)^2}\\
    <\frac{K^2C_{\max}^2}{\kappa^2(1-\kappa)^2C_{\min}^2(1-\sqrt{\gamma})^2}\fn \|\delta\|^2,
\end{equation*}
for sufficiently large $n$ and all $\lambda \in \Lambda$. This implies that the third term of $b_\lambda$ almost surely converges to zero for all $\lambda \in \Lambda$. Thus, we obtain $b_\lambda \xrightarrow{a.s.}0$ for all $\lambda \in \Lambda$. \par

Next, we evaluate $c_\lambda$. The first term of $c_\lambda$ is bounded as follows:
\begin{align*}
   & \frac{2}{n} \left|\beta_0^T\{\Sigmahat(\Sigmahat+\lambda I_p)^{+}-I_p\}\Sigma(\Sigmahat+\lambda I_p)^{+}X^T\delta\right| \\
   &\leq \frac{2}{n} \left\|X(\Sigmahat+\lambda I_p)^+\Sigma \{\Sigmahat (\Sigmahat+\lambda I_p)^+-I_p\}\beta_0\right\|\|\delta\|=(iv),
\end{align*}
for all $\lambda \in \ \Lambda$. Note that for all $\lambda \in \Lambda$,
\begin{align*}
  &\frac{1}{\sqrt{n}}\left\|X(\Sigmahat+\lambda I_p)^+\Sigma \{\Sigmahat (\Sigmahat+\lambda I_p)^+-I_p\}\beta_0\right\| \\
  &\leq \lambda_{\max}(\Sigma)\left\|\frac{1}{\sqrt{n}}X(\Sigmahat+\lambda I_p)^+\right\| \|\Sigmahat (\Sigmahat+\lambda I_p)^+-I_p\|\|\beta_0\|,
\end{align*}
and 
\begin{equation*}
    \|\Sigmahat (\Sigmahat+\lambda I_p)^+-I_p\|\leq \|\Sigmahat (\Sigmahat+\lambda I_p)^+\|+1.
\end{equation*}
For sufficiently large $n$ and all $\lambda\in \Lambda$, we almost surely have
\begin{equation*}
  \|\Sigmahat(\Sigmahat+\lambda I_p)^+\|=\sqrt{\lambda_{\max}((\Sigmahat+\lambda I_p)^+\Sigmahat^2(\Sigmahat+\lambda I_p)^+)} < \frac{KC_{\max}}{\kappa(1-\kappa)C_{\min}}.
\end{equation*}
Therefore, for sufficiently large $n$ and all $\lambda\in \Lambda$, we almost surely obtain  
\begin{equation*}
    (iv)<\frac{\sqrt{K C_f}C_{\max}^2(K C_{\max}+\kappa(1-\kappa)C_{\min})}{C_{\min}^3|1-\sqrt{\gamma}|\kappa^2(1-\kappa)^2}\frac{1}{\sqrt{n}}\|\delta\|. 
\end{equation*}
This implies that the first term of $c_\lambda$ almost surely converges to zero for all $\lambda \in \Lambda$. The numerator of the second term of $c_\lambda$ is evaluated as follows:
\begin{equation*}
    \frac{2}{n}\left|\delta^T\left(\fn XX^T+\lambda I_n\right)^{+2}X\beta_0\right|
     \leq \frac{2}{n}\|\delta\|\left\|\left(\fn XX^T+\lambda I_n\right)^{+2}X\beta_0\right\|, 
\end{equation*}
by the Cauchy--Schwarz inequality. For sufficiently large $n$ and all $\lambda \in \Lambda$, we almost surely obtain 
\begin{align*}
    &\fn \left\|\left(\fn XX^T+\lambda I_n\right)^{+2}X\beta_0\right\|^2=\fn \beta_0^TX^T\left(\fn XX^T+\lambda I_n\right)^{+4}X\beta_0\\
    &\leq \left\|\left(\fn XX^T+\lambda I_n\right)^{+4}\right\|\beta_0^T\Sigmahat\beta_0\leq \lambda_{\max}(\Sigmahat)\left\|\left(\fn XX^T+\lambda I_n\right)^+\right\|^4\|\beta_0\|^2\\
    &<\frac{ K C_{\max}^2C_f}{\kappa^4(1-\kappa)^4C_{\min}^6(1-\sqrt{\gamma})^2 }.
\end{align*}
Therefore, we almost surely obtain 
\begin{equation*}
    \left|\frac{\frac{2}{n}\delta^T(\fn XX^T+\lambda I_n)^{+2}X\beta_0}{\left(\fn\tr\left[ (\fn XX^T+\lambda I_n)^{+}\right]\right)^2}\right|<\frac{\sqrt{KC_f}C_{\max}(KC_{\max}(1-\sqrt{\gamma})^2+\lambda_{\max})^2}{\kappa^2(1-\kappa)^2 C_{\min}^3|1-\sqrt{\gamma}|}  \frac{2}{\sqrt{n}} \|\delta\|,
\end{equation*}
for sufficiently large $n$ and all $\lambda \in \Lambda$. This implies that the second term of $c_\lambda$ almost surely converges to zero for all $\lambda \in \Lambda$. Thus, we obtain $c_\lambda \xrightarrow{a.s.}0$ for all $\lambda \in \Lambda$. This completes the proof of this lemma.
\end{proof}

\begin{proof}[Proof of Theorem \ref{thm1}]
    By Lemmas \ref{lem1}--\ref{lem4}, we obtain
\begin{align*}
    R(\lambda)-\gcv(\lambda)&= R_b(\lambda)-\frac{\gcv_b(\lambda)}{\gcv_d(\lambda)}+R_v(\lambda)-\frac{\gcv_v(\lambda)}{\gcv_d(\lambda)}+ R_c(\lambda)-\frac{\gcv_c(\lambda)}{\gcv_d(\lambda)}\\
    &\xrightarrow{a.s.}0,
\end{align*}
for all $\lambda \in \Lambda$. Therefore, the pointwise convergence holds on $\Lambda$. To show the uniform convergence on $\Lambda$, it is sufficient to show that $R(\lambda)$, $\gcv(\lambda)$, and their derivatives are bounded for sufficiently large $n$ and all $\lambda \in \Lambda$, as in \cite{pwrt2021}. \par

We first evaluate $\gcv(\lambda)$. As in \cite{pwrt2021}, let us define 
\begin{align*}
    r_n(\lambda)&=\fn Y^T\left(\fn XX^T+\lambda I_n\right)^{+2}Y,\\ 
    v_n(\lambda)&=\left(\tr\left[\left(\fn XX^T+\lambda I_n\right)^{+}\fn \right]\right)^2.
\end{align*}
Because we are defining $\gcv(0):=\lim_{\lambda \rightarrow 0}\gcv(\lambda)$, we can write 
\begin{equation*}
    \gcv(\lambda)=\frac{r_n(\lambda)}{v_n(\lambda)},
\end{equation*}
for all $\lambda \in \Lambda$. Note that 
\begin{equation*}
    \fn \|Y\|^2=\fn \sum_{i=1}^nf(x_i)^2 +\frac{2}{n}\sum_{i=1}^nf(x_i)\epsilon_i +\fn\|\epsilon\|^2,
\end{equation*}
and the strong law of large numbers yields $\fn \sum_{i=1}^nf(x_i)^2\xrightarrow{a.s.}\expect[\{f(x_i)\}^2]\leq C_f$, 
$\fn \sum_{i=1}^nf(x_i)\epsilon_i\xrightarrow{a.s.}\expect[f(x_i)\epsilon_i]=0$, and 
$\fn \|\epsilon\|^2 \xrightarrow{a.s}\sigma^2$. They imply that $\fn \|Y\|^2$ is almost surely bounded for sufficiently large $n$. Thus, we almost surely obtain 
 \begin{equation*}
        |\gcv(\lambda)|\leq \frac{(\lambda_{\max}(\fn X^TX)+\lambda)^2}{(\lambda_{\min}(\fn XX^T)+\lambda)^2}\fn \|Y\|^2< \frac{(KC_{\max}(1-\sqrt{\gamma})^2+\lambda_{\max})^2}{\kappa^2(1-\kappa)^2 C_{\min}^2(1-\sqrt{\gamma})^4} \fn \|Y\|^2 \leq C,
    \end{equation*}
for sufficiently large $n$ and all $\lambda \in \Lambda$. We next evaluate the derivative of $\gcv(\lambda)$, which is given by
\begin{equation*}
    \gcv'(\lambda)=\frac{r_n'(\lambda)v_n(\lambda)-r_n(\lambda)v_n'(\lambda)}{v_n(\lambda)^2}.
\end{equation*}
Simple calculation yields 
\begin{align*}
    |r_n'(\lambda)|&\leq \frac{2}{n} \frac{\|Y\|^2 }{|\lambda_{\min}\left(\fn XX^T\right)+\lambda|^3},\\
    |v_n'(\lambda)|&=\frac{2}{n^2}\left|\sum_{k=1}^n \frac{1}{\lambda_k+\lambda}\sum_{k=1}^n \frac{1}{(\lambda_k+\lambda)^2}\right|\leq \frac{2}{|\lambda_{\min}\left(\fn X X^T)+\lambda\right|^3}.
\end{align*}
Therefore, we almost surely obtain  
\begin{equation*}
    |\gcv(\lambda)'|\leq \frac{\left(\lambda_{\max}(\fn XX^T)+\lambda\right)^2}{|\lambda_{\min}(\fn XX^T)+\lambda|^5}\frac{4}{n}\|Y\|^2<\frac{4(KC_{\max}(1-\sqrt{\gamma})^2+\lambda_{\max})^2}{\kappa^5(1-\kappa)^5C_{\min}^5(1-\sqrt{\gamma})^{10}}\fn \|Y\|^2 \leq C,
\end{equation*}
for sufficiently large $n$ and all $\lambda \in \Lambda$. \par 

We next evaluate $R(\lambda)$ and $R(\lambda)'$. We can write
\begin{equation*}
    R(\lambda)=\expect(y_0^2)-2\expect[f(x_0)x_0^T]\betahat_\lambda+\betahat_\lambda \Sigma \betahat_\lambda. 
\end{equation*}
We obtain 
\begin{equation*}
    \expect(y_0^2)=\expect[f(x_0)^2]+\sigma^2\leq C_f+\sigma^2.
\end{equation*}
The Cauchy--Schwarz inequality yields $|\expect[f(x_0)x_0^T]\betahat_\lambda|\leq \|\expect[f(x_0)x_0]\|\|\betahat_\lambda\|$. We obtain
\begin{align*}
\|\expect[f(x_0)x_0]\|=\sup_{\|a\|=1}|\expect[f(x_0)a^Tx_0]|\leq \sqrt{\expect[f(x_0)^2]} \sqrt{\sup_{\|a\|=1}\expect[(a^Tx_0)^2]} \leq \sqrt{C_fC_{\max}}.
\end{align*}
For sufficiently large $n$ and all $\lambda \in \Lambda$, we almost surely obtain 
\begin{align*}
  \|\betahat_\lambda\|^2&=\frac{1}{n^2}Y^TX(\Sigmahat+\lambda I_p)^{+2}X^TY\leq \lambda_{\max}\left(\fn X(\Sigmahat+\lambda I_p)^{+2}X^T\right) \fn \|Y\|^2 \\
  &<\frac{KC_{\max}}{\kappa^2(1-\kappa)^2 C_{\min}^2(1-\sqrt{\gamma})^2} \fn \|Y\|^2\leq C.
\end{align*}
This implies that the second term and the third term in $R(\lambda)$ are almost surely bounded for sufficiently large $n$ and all $\lambda \in \Lambda$. Thus, $R(\lambda)$ is almost surely bounded for sufficiently large $n$ and all $\lambda \in \Lambda$.\par $R(\lambda)'$ is given by 
\begin{equation*}
    R(\lambda)'=2\expect[f(x_0)x^T_0](\Sigmahat+\lambda I_p)^{+2}\fn X^TY-\frac{2}{n^2}Y^TX(\Sigmahat+\lambda I_p)^+\Sigma(\Sigmahat+\lambda I_p)^{+2}\fn X^TY.
\end{equation*}
We almost surely obtain
\begin{align*}
&\frac{2}{n}|\expect[f(x_0)x^T_0](\Sigmahat+\lambda I_p)^{+2} X^TY|\leq \frac{2}{n}\|\expect[f(x_0)x_0]\| \|(\Sigmahat+\lambda I_p)^{+2}X^TY\|\\
&\leq 2 \sqrt{C_fC_{\max}} \left\|\frac{1}{\sqrt{n}}X(\Sigmahat+\lambda I_p)^{+2}\right\| \frac{1}{\sqrt{n}}\|Y\|<\frac{2\sqrt{C_fK}C_{\max}}{\kappa^2(1-\kappa)^2C_{\min}^2|1-\sqrt{\gamma}|^3} \frac{1}{\sqrt{n}}\|Y\|\leq C,
\end{align*}
and
\begin{align*}
&\frac{2}{n^2}|Y^TX(\Sigmahat+\lambda I_p)^+\Sigma(\Sigmahat+\lambda I_p)^{+2} X^TY|\leq \frac{1}{n^2}\|Y\| \|X(\Sigmahat+\lambda I_p)^+\Sigma(\Sigmahat+\lambda I_p)^{+2} X^TY\|\\
&\leq \lambda_{\max}(\Sigma)\fn \|Y\|^2 \left\|\frac{1}{\sqrt{n}} X(\Sigmahat+\lambda I_p)^+\right\| \left\|\frac{1}{\sqrt{n}}X(\Sigmahat+\lambda I_p)^{+2}\right\|\\
&<\frac{KC_{\max}^2}{\kappa^3(1-\kappa)^3 C_{\min}^3 (1-\sqrt{\gamma})^4}\fn \|Y\|^2 \leq C,
\end{align*}
for sufficiently large $n$ and all $\lambda \in \Lambda$. Therefore, $R(\lambda)'$ is almost surely bounded for sufficiently large $n$ and all $\lambda \in \Lambda$. \par

We have shown that $\gcv(\lambda)-R(\lambda)$ and its derivative with respect to $\lambda$ are almost surely bounded for sufficiently large $n$ and all $\lambda \in \Lambda$. This implies that $\gcv(\lambda)-R(\lambda)$ is equicontinous on $\Lambda$. Thus, the Arzela--Ascoli theorem implies the uniform convergence on $\Lambda$. This completes the proof of this theorem. 
\end{proof}

\bibliographystyle{apalike} 
\bibliography{ridge}
\end{document}